\documentclass[11pt, reqno]{amsart}

\usepackage{amsmath,amssymb,latexsym}
\usepackage{graphicx,xcolor,mathrsfs}
\usepackage{cite}

\usepackage[font={scriptsize}]{caption}
\usepackage{extarrows}

\newcommand{\bigO}{\mathcal{O}}
\newcommand{\diff}{\mathrm{d}}

\newtheorem{theorem}{Theorem}[section]
\newtheorem{lemma}[theorem]{Lemma}

\newtheorem{remark}[theorem]{Remark}

\newtheorem{corollary}[theorem]{Corollary}
\newtheorem*{main-theorem}{Main Theorem}
\newtheorem*{remark*}{Remark}

\numberwithin{equation}{section}

\def \div {\mathrm{div}}
\let\f=\frac
\def \p {\partial}
\def \pt {\partial_t}
\def \na {\nabla}
\def\vr {\varrho}
\def \na {\nabla}
\def \lk {\Lambda^k}
\def \cE {\mathcal{E}}
\def \ls {\Lambda^s}

\usepackage[colorlinks=true]{hyperref}
\hypersetup{urlcolor=blue, citecolor=red, linkcolor=blue}

\begin{document}
	
	\allowdisplaybreaks
	
	\title[]{Large time existence of Euler-Korteweg equations and two-fluid Euler-Maxwell equations with vorticity}

\author{Changzhen Sun}


\address{Universit\' e Paris-Saclay, CNRS, Laboratoire de Math\'  ematiques d'Orsay, 91405 Orsay, France.}
\email{changzhen.sun@u-psud.fr}

	\begin{abstract}
The aim of this paper is to study the influence of the vorticity on the existence time in fluid  systems for which global
smoothness and decay is known in the case of small irrotational data. We focus on
two examples: the Euler-Korteweg system and the two-fluid Euler Maxwell system.
We prove that the lower bound of the lifespan of these systems is no less than
the inverse of the $H^s$ $(s>5/2)$ norm of the rotational part of the initial velocity.
Our approach is based on energy estimates and the fast time decay results of global solutions to these systems with irrotational initial data.
	

	\end{abstract}

	\maketitle

\section{Introduction}

In this paper, we are concerned with the well-posedness of 3-d compressible Euler-Korteweg system which reads

\begin{eqnarray}\label{EK}
\left\{
\begin{array}{l}
\displaystyle
\partial_t \rho + \div( \rho u)=0,\\
\displaystyle \rho \partial_t u+ \rho u\cdot\nabla u +
\nabla P(\rho)-\rho \nabla \bigg(K(\rho)\Delta \rho +\f{1}{2}K'(\rho)|\nabla\rho|^2\bigg)=0, \\
\displaystyle u|_{t=0} =u_0 ,\rho|_{t=0}=\rho_0.
\end{array}
\right.
\end{eqnarray}
where $\rho,u$
are the density and velocity of the fluid,
$P(\rho)$ is the pressure and is assumed to obey the so-called $\gamma$-law, that is $P(\rho)=A\rho^{\gamma},A>0$.
$K(\rho)$ is the Korteweg tensor, which takes the capillary effects into account.

As the modifications of compressible Euler equation through the adjunction the Korteweg stress tensor, Euler-Korteweg system is
a mathematical model arising from  hydrodynamics and quantum hydrodynamics. In fact, in hydrodynamics, it can be used to describe at interface the flow of capillary flows, for example, a liquid-vapor mixture. Moreover, when $K(\rho)=\f{c}{\rho}, P(\rho)=\f{1}{2}\rho^2$, the Euler-Korteweg system can be transformed formally by the so-called Madelung transform\cite{MR2979973} $\psi=\sqrt{\rho}e^{i\phi}, u=\nabla\phi$ to the Gross-Pitaevskii
equation which is a very important equation in geometric optics and quantum mechanics.

In the last two decades, some interesting results concerning the well-posedness of compressible Euler-Korteweg have been obtained. For the  cauchy problem of
Euler-Korteweg in one dimension, local well-posedness for smooth perturbations of travelling profiles was established by Benzoni-Gavage, Danchin and Descombes
\cite{MR2226932} by using Lagrangian coordinate.
Later on, in \cite{MR2354691}, the same authors extend the
result to multi-dimension 
by reformulating the system into a nonlinear degenerate Schr\"odinger equation incorporated with the 'gauge' technique (by introducing some 'gauge' function to recover some cancellations in order to avoid losing derivatives). More precisely, for
$d(d\geq 1)$ dimensional Euler-Korteweg system, they proved the local existence under $H^{s+1}\times H^s$ $(s>d/2+1)$ perturbations
 to the stationary solution $(\bar{\rho},0)(\bar{\rho}>0)$.
 Recently, Bezoni-Gavage and Chiron \cite{MR3763346} establish some uniform estimate in several different wave regimes and justify the asymptotic limit. Concerning the global well-posedness, it is shown by Audiard and Haspot \cite{MR3613503} that
3-D Euler-Korteweg system admits global small \textit{irrotational} solutions, by enforcing the so-called 'space-time resonance' method which turns out to be the efficient tools to get the global existence of some models that admits dispersive properties with critical
nonlinearities. One can refer to
\cite{germain2010space,MR2993751,MR2360438}.
Regarding to the large time existence with nontrivial vorticity, 
Audiard proved in \cite{audiard2019time} that the lifespan $T_{\star}$ of 3-D Euler-Korteweg system is no less than the inverse of the size of the rotational part in some suitable weighted space. His strategy is to study rigorously the highly coupled system composed by the equations of the 'rotational' and 'irrotational part' of the velocity. Nevertheless,
in this process, one needs to deal with the complicated interactions between 'rotational' parts and 'irrotational' part.
The first aim of this paper is to give an alternative approach for the lifespan estimate of Euler-Korteweg system with vorticity, where merely energy estimates are used. We remark also that our proof does not need the localization assumption on
the 'rotational' part  of initial velocity.

We denote $\mathcal{P}$ the Leray projector that maps a vector in $L^2(\mathbb{R}^3)^3$ to its divergence free part and  $\mathcal{P}^{\perp}=Id-\mathcal{P}$ the 'curl-free' projector.

The following is the main results:

\begin{theorem}
\label{thmek0}
Suppose that the Korteweg tensor $K(\rho)$ is smooth and satisfies:
$K(\rho)\geq K_0>0$ for $\bar{\rho}/2\leq \rho\leq 3\bar{\rho}/2$.
 There exists three constants $\delta_1, \epsilon_1>0$ small, and $N$ large. If the initial datum $(\rho_0-\bar{\rho},u_0)$
  satisfies the following:
  \begin{equation}
	\begin{aligned}
	\|\mathcal{P} u_0\|_{H^s}<\epsilon_{1},\nonumber
	\end{aligned}
	\end{equation}
   \begin{equation}
	\begin{aligned}
 & \|\mathcal{P}^{\perp}u_0\|_{H^N}+\|\rho_0-\bar{\rho}\|_{H^{N+1}}+\|x(\rho_0-\bar{\rho},\mathcal{P}^{\perp} u_0)\|_{L^2}\\
&\qquad\qquad +\|u_0\|_{W^{6,1}}+\|\rho_0-\bar{\rho}\|_{W^{7,1}}\leq \delta_1, \nonumber
  \end{aligned}
	\end{equation}
where $5/2<s\leq 3.$ Then there exists $T_{\epsilon_1}\gtrsim \epsilon_1^{-1}$ such that the Euler-Korteweg equation \eqref{EK}
 has a unique solution and  $$(\rho-\bar{\rho},u)\in C([0,T_{\epsilon_1}],H^{s+1}(\mathbb{R}^3)\times H^s(\mathbb{R}^3)).$$

In addition, if
 $\mathcal{P}u_0\in H^{N}$, then the solution  $$(\rho-\bar{\rho},u) \in C([0,T_{\epsilon_1}],H^{N+1}(\mathbb{R}^3)\times H^{N}(\mathbb{R}^3)).$$
 with exponential growth: for any $0\leq t\leq T_{\epsilon_1}$,
 $$\|(\rho-\bar{\rho},u)(t)\|_{H^{N+1}\times H^N}\lesssim e^{ct}\|(\rho-\bar{\rho},u)(0)\|_{H^{N+1}\times H^N}.$$
\end{theorem}


To prove Theorem \ref{EK},
a natural attempt is, as in \cite{audiard2019time}, to study the highly coupled system
 by considering the 'rotational' parts $\mathcal{P} u$  and 'irrotational' parts $\mathcal{P}^{\perp} u=u-\mathcal{P} u$.
 However, when one tries to extend the lifespan of $\mathcal{P} u$  to $1/\|\mathcal{P}^{\perp} u_0\|_{H^s}$ $(s>5/2),$ it is necessary to prove  that the 
irrotational part of velocity $\mathcal{P}^{\perp} u$ enjoys the integrable time decay. However, when using 'space-time resonance' method to  perform decay estimate for $\mathcal{P} ^{\perp}u$,  one needs to study the 'dispersive$\times$ vorticity' interactions which bring a lot of extra work. On the other hand,
since one need that $\mathcal{P} u$ is in some weighted space in order to prove its decay property, this 'dispersive$\times$ vorticity' interactions will forces us to  assume that the $\mathcal{P} ^{\perp}u$ also belongs to some weighted space.   In the following, we propose a shorter approach that do not require the rotational part $\mathcal{P} ^{\perp}u$ lies in any weighted space.

To explain the main ideas, we will restrict ourselves to the more abstract setting. Consider a system:
\begin{equation}\label{toy}
\left\{
\begin{array}{l}
\displaystyle
      \partial_t U +J L U =U'\cdot\nabla U\\
      \displaystyle  U|_{t=0}=U_0\\
\end{array}
\right.
\end{equation}
where $U(t,x)=(U_1,U'):\mathbb{R}_{+}\times\mathbb{R}^3\rightarrow \mathbb{R}^4$, is a four-elements vector function,  $J$ is a skew symmetric differential matrix, $L$ is self-adjoint and positive in some suitable space in the sense that $(Lu,v)_{L^2}=(u,Lv)_{L^2}, (LU,U)_{L^2}\geq \|U\|_{L^2}^2$.
For example $$J=\left(
  \begin{array}{cc}
    0&\div\\
    \na&0\\
  \end{array}
\right),
\qquad
L=\left(
  \begin{array}{cc}
 Id-\Delta &0\\
   0 &Id_{3\times 3}\\
  \end{array}
\right),$$
is the Euler-Korteweg type equations (the simplified case that the term $\rho\div u$ is dropped in $\eqref{EK}_1$ and $P(\rho)=\frac{1}{2}\rho^2, K(\rho)=1$ in
$\eqref{EK}_2$ is assumed)
while
$$J=\left(
  \begin{array}{cc}
0 &\div\\
   \na & 0\\
  \end{array}
\right),
\qquad
L=\left(
  \begin{array}{cc}
 Id+(-\Delta)^{-1} &0\\
   0 &Id_{3\times 3}\\
  \end{array}
\right),$$
is the one fluid Euler-Poisson type equations. 

We suppose firstly that for the curl-free ($\mathrm {curl} U'_{0}=0$) smooth initial datum, there exists global solutions in some Sobolev space $H^N$ (where $N$ is large enough) which decays fast enough to 0 as the time goes to infinity. More precisely, we suppose that
$\|U\|_{W^{4,\infty}}\lesssim (1+t)^{-{\alpha}}$ with $\alpha>1$.
Now, we want to analyze the large time existence of system \eqref{toy} with general
(not necessarily curl-free) smooth initial data.
Our strategy is to split the system \eqref{toy}
into two systems, with initial data $(U_1,\mathcal{P}^{\perp}U_{0}')$ and $(0,\mathcal{P}{U_0'})$. 
To be more concrete, we write $U=W+V$, where $W$ solves \eqref{toy} with initial data $(0,\mathcal{P}^{\perp}U_{0}')$, and $V$ satisfies the equation:
\begin{equation}\label{toyper}
\left\{
\begin{array}{l}
\displaystyle
      \partial_t V +J L V =(V'+W')\cdot\nabla V+ V'\cdot \nabla W =: F(V,W)\\
      \displaystyle  W|_{t=0}=(0,\mathcal{P}U'_0)\\
\end{array}
\right.
\end{equation}
To study the long time existence of \eqref{toyper}, it suffices for us to get appropriate a priori energy estimates.
Let us define energy functional $$E_s(t)=
\int_{\mathbb{R}^3} 
\ls V\cdot L\ls V (t) \diff x$$ where $\Lambda=\sqrt{1-\Delta}$ and $s>5/2$.
Taking $\ls$ on system \eqref{toyper}, and testing $L\ls V$, we then get
\begin{equation*}
\partial_t E_s\leq \int_{\mathbb{R}^3} \ls F(V,W)\cdot L \ls V \diff x
\end{equation*}
which yields by commutator estimate, if $L$ is lower order operator (for example Euler-Poisson system), we could get that:
$$\partial_t E_s\lesssim E_s^{\frac{3}{2}}+\|W\|_{W^{s+1,\infty}}E_s\lesssim
E_s^{\f{3}{2}}+(1+t)^{-\alpha}E_s$$
 from which, one deduce by the Gr\"onwall inequality and continuation arguments,
 that, there exists solutions for system \eqref{toyper} in $C([0,T),H^s)$ for $T\gtrsim 1/\|\mathcal{P}U'_0\|_{H^s}$.
However, when $L$ is higher order (for example Euler-Korteweg type), direct energy estimate will inevitably lose derivatives. In this case, the 'gauge' technique used in \cite{MR2354691}
need to be employed. 

We would like to mention that this strategy is inspired by the former work of the author with Rousset \cite{rousset2019stability} where the uniform stability for Navier-Stokes-Poisson system in the inviscid limit is established. It turns out that this approach is flexible for many models that admit global solutions with integrable time decay under the irrotational initial perturbation to equilibria, one could consider Euler-Poisson, Euler-Maxwell (one fluid, two-fluid)...Moreover, the method proposed in this paper will 
simplify the proof to large extend when the 'space-time resonance' of phase function is difficult to analyze, since one do not need to
take care of the new resonances arising from the 'dispersive$\times$vorticity' interactions.
We will prove the similar results to Theorem \ref{thmek0} for 'two-fluid' Euler-Maxwell equation \eqref{EM0}
 which is new.
 Note in \cite{MR3742601},
 Ionescu and Lie prove the long time existence for one-fluid Euler-Maxwell equations, although their method is likely to be adapted to prove the
 similar results for 'two-fluid' case,  the proof will be much sophisticated since the 'space-time resonance' is harder to analyze than the 'one-fluid' case.

\textbf{Organisation of the paper}
We will use the strategy stated above to prove Theorem \ref{EK} in Section 2.
We then prove the similar result for 'two-fluid' Euler-Maxwell equations in Section 3. Finally, we recall some useful  lemmas in appendix.

\section{Proof of Theorem \ref{thmek0}}


As explained in the Introduction, we split the original system \eqref{EK} into two systems by letting
\begin{equation*}
\begin{aligned}
&\rho=(\bar{\rho}+\varrho)+n=:\tilde{\rho}+n,\\
&u=w+v,
\end{aligned}
\end{equation*} 	
such that the unknowns $(\varrho,w)$ satisfy the system
\begin{eqnarray}\label{EK1}
\left\{ \begin{array}{ll}
	\partial_t \tilde{\rho} +\mathrm{div}( \tilde{\rho} w)=0,\\
    \partial_t w+  w\cdot\nabla w +
g(\tilde{\rho})	\nabla\tilde{\rho} 
	=\nabla(K(\tilde{\rho})  \Delta \tilde{\rho}+\frac{1}{2}K'(\tilde{\rho})|\tilde{\rho}|^2), \\
	 w|_{t=0} =\mathcal{P}^{\perp}u_0,\tilde{\rho}|_{t=0}=\rho_0,
\end{array}
\right.
\end{eqnarray}
and the unknowns $(n,v)$ satisfy the system
\begin{eqnarray}\label{EK2}
\left\{ \begin{array}{ll}
	\partial_t n +\mathrm{div}( \rho v+nw)=0,\\
	\displaystyle \partial_t v+ (w+v)\cdot\nabla v+v\cdot\nabla w +g(\bar{\rho})\na n\\
	\displaystyle \qquad\qquad\qquad=\nabla\big(K(\rho)\Delta n +\frac{1}{2}K'({\rho})\na n\cdot \na (n+2\tilde{\rho})\big)+F,
 \\
	 v|_{t=0} =\mathcal{P}u_0,\ n|_{t=0}=0.
\end{array}
\right.
\end{eqnarray}
where we denote
$g(\rho)=P'(\rho)/\rho$ and
\begin{equation}\label{defofF}
\begin{aligned}
F&=-\big(g(\rho)-g(\bar{\rho})\big)\nabla n
-\big(g(\rho)-g(\tilde{\rho})\big)\na\tilde{\rho}\\
&\qquad\qquad+\nabla \big[\big(K(\rho)-K(\tilde{\rho})\big)\Delta\tilde{\rho}+\f{1}{2}\big(K'(\rho)-K'(\tilde{\rho})\big)\nabla |\tilde{\rho}|^2\big].
\end{aligned}
\end{equation}

For the system $\eqref{EK1}$, we recall the global existence result established in \cite{MR3613503},
\begin{theorem}\label{thmek1}
	 {Theorem 2.2  ,\cite{MR3613503}} :
	
	 There exists  a small number $\delta>0$ and a large integer $N\geq 6$, such that if the initial data satisfy
	\begin{equation*}
	\begin{aligned}
	&\|\mathcal{P}^{\perp}u_0\|_{H^N}+\|\rho_0-\bar{\rho}\|_{H^{N+1}}+\|x(\rho_0-\bar{\rho},\mathcal{P}^{\perp} u_0)\|_{L^2}\\
	&\quad+\|u_0\|_{W^{6,1}}+\|\rho_0-\bar{\rho}\|_{W^{7,1}}\leq \delta_1,
	\end{aligned}
	\end{equation*}
	then the Cauchy problem \eqref{EK1} admits a global solution $(\varrho,u)$ in\\ $C\left([0,\infty),H^{N+1}\times H^N\right)$ satisfying for any $t\geq 0,$
	\begin{align}\label{bound for varrho}
	|\varrho(t)|\leq \frac{1}{4} \min\{\bar{\rho},1\},
	\end{align}
	Moreover, there exist $\alpha>1$ and $C>1$ such that
	\begin{equation}\label{est:decay}
	\begin{aligned}
	\sup_{t\geq 0} \big(\|u(t)\|_{H^N}+\|\varrho\|_{H^{N+1}}+(1+t)^{\alpha}\|(\varrho,u)\|_{W^{6,\infty}}\big)\leq C\delta_1.
	\end{aligned}
	\end{equation}
\end{theorem}

To prove the first part of Theorem \ref{thmek0}, it suffices to show the following:

\begin{theorem}\label{thmek2}
Suppose ($\vr,w$) are global solutions to the system \eqref{EK1} given by Theorem \ref{thmek1}. There exists $\epsilon_1$ small enough, if
$\|\mathcal{P}u_0\|_{H^s}\leq \epsilon_1$
$(\f{5}{2}< s\leq 3)$,
then one can find some $T_{\epsilon_1}\gtrsim \epsilon_{1}^{-1}$ such that the system \eqref{EK2}  has a unique solution and  $(n,v)\in C([0,T_{\epsilon_1}],H^{s+1}(\mathbb{R}^3)\times H^s(\mathbb{R}^3)).$
\end{theorem}

{\bf Proof of Theorem \ref{thmek2}}
The Cauchy problem \eqref{EK2} is well-posed in $C([0,T),H^s),s>5/2$ for some positive $T>0$ (e.g. \cite{MR2354691}). We are thus left to show the lifespan of \eqref{EK2} has the order of $\bigO(\epsilon_1^{-1})$ by developing an a priori estimate.

The direct energy estimate will cause the loss of derivatives due to the lack of dissipation and
the presence of the high order term $\nabla(K(\rho)\Delta n)$ in \(\eqref{EK2}_2\). To get round this difficulty, one needs to introduce certain weight function (which is called 'gauge' function \cite{MR2354691}) coherent to the energy functional to eliminate this kind of terms.

Denote $\Lambda$ the Fourier multiplier with symbol $\langle\xi\rangle=\sqrt{1+|\xi|^2}$.
We will work on the following energy functionals:
	\begin{equation*}
	\begin{aligned}
\mathcal{E}_{s}(t)=
\frac{1}{2}\int \phi_s(\rho)\big(g(\bar{\rho})|\ls n|^2+K(\rho)|\Lambda^{s}\nabla n|^2+\rho |\Lambda^{s}v|^2\big)\,\diff x, 
\end{aligned}
\end{equation*}
where the 'gauge' function $\phi_s(\rho)=(\rho K(\rho))^{\frac{s}{2}}.$ The role of this gauge function is to avoid the loss of derivative when $\mathcal{P} v=0$. We shall comment that $\phi_s(\rho)$ depend on both $n$ and $\varrho$ (recall that $\rho=\bar{\rho}+\vr+n$).

Taking the time derivative on functional $\mathcal{E}_s$ and using the equations
\eqref{EK2},
one obtains
\begin{equation}\label{sec2_eq:1}
\begin{aligned}
\frac{\diff}{\diff t} \mathcal{E}_{s}(t)&=
\int g(\bar{\rho})\partial_t \phi_s |\ls n|^2+\partial_t(K\phi_s(\rho))|\Lambda^{s}\nabla n|^2+\partial_t(\rho \phi_s(\rho) )|\Lambda^{s}v|^2\,\diff x
\\
&\quad-\int g(\bar{\rho})\phi_s(\rho)\ls n\ls \div(nw)- \rho \phi_s(\rho)
{\Lambda^{s} v}\cdot \Lambda^{s}F\diff x\\
&\quad-
g(\bar{\rho})\int \phi_s(\rho) \big(\ls n
\ls \mathrm{div}(\rho v)+
\rho
{\Lambda^{s} v}\Lambda^s\nabla n\big)
\diff x\\
&\quad -
\int \rho \phi_s(\rho)
{\Lambda^{s} v} \Lambda^{s}\big((w+v)\cdot\nabla v+v\cdot\nabla w\big)+K\phi_s(\rho){\Lambda^{s}\nabla n}\Lambda^{s} \nabla\mathrm{div}(nw)\diff x
\\
&\quad-
\int K\phi_s(\rho)
{\Lambda^{s}\nabla n}
\Lambda^{s}\nabla\mathrm{div}(\rho v)\\
&\qquad\qquad -\rho\phi_s(\rho) 
{\Lambda^{s} v}\cdot \Lambda^{s} \nabla\big[K(\rho)\Delta n+\frac{1}{2}K'(\rho)\na n\cdot\na (n+2\tilde{\rho})\big]\,\diff x \\
&=\colon I_1+I_2+I_3+I_4+I_5.
\end{aligned}
\end{equation}

Among the five terms \(I_i, 0\leq i\leq 5\), the term \(I_5\) is the most difficult one since it involves the loss of derivatives, we will postpone handling it later on and first deal with the other four easier terms.

For $I_1$, by rewriting \(\eqref{EK}_1\) as $\partial_t\rho=-\mathrm{div}\big(\rho(w+v)\big)$, one may estimate
\begin{align}\label{sec2_eq:2}
|I_1|\lesssim \|\rho\|_{W^{1,\infty}}(\|w\|_{W^{1,\infty}}+\|v\|_{W^{1,\infty}})\|(\nabla n,v)\|_{H^s}^2.
\end{align}
Note that we have used the fact that $K(\rho)$ and $\phi_s(\rho)$ are smooth functions and the a priori  assumption $\bar{\rho}/2\leq \rho\leq 3\bar{\rho}/2.$

For $I_2,$ by H\"older inequality, product estimate \eqref{appendix:2.5} and Corollary $\ref{esofF},$
\begin{equation}
   \begin{aligned}
       & |I_2|\lesssim (\|n\|_{W^{1,\infty}}+\|w\|_{W^{s+1,\infty}})\|(n,\na n)\|_{H^{s}}^2
        +\|v\|_{H^s}\|F\|_{H^s}\\
       &\quad \lesssim (\|(n,\nabla n)\|_{H^s}+\|(w,\vr)\|_{W^{s+3,\infty}})\|(n,\na n,v)\|_{H^{s}}^2.
    \end{aligned}
\end{equation}

  We remind the reader that
  in order to make use of the fast decay property of $\|(\vr,w)(t)\|_{L_x^{\infty}},$
  we shall always attribute $L_x^{\infty}$ norm on $(\vr,w)$ when we estimate the product terms that could be considered roughly as $\vr n, w n$.

For $I_3$, integrating by parts, applying product estimates \eqref{appendix:2.5} and commutator estimates \eqref{appendix:3},\eqref{appendix:4},  
we get:
\begin{equation}\label{sec2_eq:3}
\begin{aligned}
I_3&=g(\bar{\rho}) \int \nabla (\phi_s(\rho))
{\Lambda^{s} n}\Lambda^{s}(\rho v)\,\diff x+g(\bar{\rho})\int \phi_s(\rho)
{\Lambda^{s}\nabla n}[\Lambda^{s},\rho]v\,\diff x\\
&\lesssim \big(\|(n,v)\|_{W^{1,\infty}}+\| (\varrho,w)\|_{W^{s,\infty}}\big)\|(n,\nabla n,v)\|_{H^s}^2.
\end{aligned}
\end{equation}

More precisely, writing \([\Lambda^{s},\rho]v=[\Lambda^{s},\varrho]v+[\Lambda^{s},n]v\),
we use \eqref{appendix:3} to get
\begin{align*}
\|[\Lambda^{s},\varrho]v\|_{L^2}\lesssim \|\varrho\|_{W^{s,\infty}}\|v\|_{H^{s-1}},
\end{align*}
and use \eqref{appendix:4} to obtain
\begin{align*}
 \|[\Lambda^{s},n]v\|_{L^2}\lesssim \|\nabla n\|_{L^{\infty}}\|v\|_{{H}^{s-1}}+\|v\|_{L^{\infty}}\|n\|_{{H}^{s}}.
\end{align*}

In the same fashion as $I_3$, the term  $I_4$ can be handled by
\begin{equation}\label{sec2_eq:4}
\begin{aligned}
I_4&=\f{1}{2}\int\mathrm{div}\big(\rho\phi_s(w+v)\big)\big|\Lambda^{s} v\big|^2
+\mathrm{div}\big(K\phi_s w\big)\big|\Lambda^{s} \nabla n\big|^2\diff x\\
&\quad-\int \rho\phi_s\Lambda^{s} v\cdot\big(\big[\Lambda^{s},w+v\big]\nabla v+\Lambda^{s}(v\cdot \nabla w)\big)\,\diff x\,\\
&\quad-\int K\phi_s(\rho) {\Lambda^{s} \nabla n}\big(\Lambda^{s}\nabla(n\mathrm{div} w)+[\Lambda^{s} \nabla, w]\nabla n\big)\,\diff x\\
&\lesssim\big(\|(n,v)\|_{W^{1,\infty}}
+\| (\varrho,w)\|_{W^{s+2,\infty}}\big)\|(n,\na n,v)\|_{H^s}^2.
\end{aligned}
\end{equation}

To estimate $I_5,$ it will be helpful to extract the principle term of $\Lambda^{s} \nabla\big[K(\rho)\Delta n+\frac{1}{2}K'(\rho)\na n\cdot\na (n+2\tilde{\rho})\big]$ that may lose derivatives.
At first,
\begin{equation}\label{sec2:eqK1}
\begin{aligned}
&\ls \nabla \big(\f{1}{2}K'(\rho)\na n\cdot\na (n+2\tilde{\rho})\big)\\
&\quad=\f{1}{2} [\ls \nabla,K'(\rho)](\na n \cdot\na (n+2\tilde{\rho}))+\f{1}{2}K'(\rho)\ls\na(\na n \cdot\na (n+2\tilde{\rho}))\\
&\quad=l.o.t+K'(\rho)\na\rho \cdot\ls \nabla\nabla n
\end{aligned}
\end{equation}
where the notation $l.o.t$ stands for terms which can be controlled by
\begin{equation*}
\begin{aligned}
\|l.o.t\|_{L^2}\lesssim (\|(n,v)\|_{H^s}+\|\varrho\|_{W^{s+2,\infty}})\|(\nabla n,v)\|_{H^{s}}.
\end{aligned}
\end{equation*}

Similarly, we have that:
\begin{equation}\label{sec2:eqK2}
\begin{aligned}
\ls \nabla (K(\rho)\Delta n)&=[\ls,\nabla K]\Delta n+K\ls \nabla\Delta n+\nabla K\ls \Delta n+[\ls, K]\nabla\Delta n\\
&=l.o.t+K\ls \nabla\Delta n+ K'\ls \Delta n\nabla\rho+sK'\nabla\rho\cdot\ls\na\na n
\end{aligned}
\end{equation}
Note that we have used the commutator estimate \eqref{appendix:7}-\eqref{appendix:8} to get that:
\begin{equation}
\begin{aligned}
 &[\Lambda^s,K]\nabla\Delta n-sK'\nabla\rho\ls\na\na n\\
 &=[\Lambda^s,K]\nabla\Delta n-\frac{1}{\mathrm{i}}\{\langle\xi\rangle^s,K\}(D)\nabla\Delta n-sK'\nabla\rho \Lambda^{s-2}\nabla\nabla n\\
& =l.o.t
\end{aligned}
\end{equation}
where we used Poisson bracket:
$\{a,b\}=\partial_{\xi}a\cdot\partial_x b-\partial_{\xi}b \cdot\partial_x a.$
We are now in position to estimate $I_5.$
In view of \eqref{sec2:eqK1} and \eqref{sec2:eqK2},
we have by integrating by parts that:
\begin{equation}\label{I5}
\begin{aligned}
I_5&=\int \rho\phi_s\Lambda^s v\big(K\ls\nabla\Delta n+K'\ls\Delta n\nabla \rho+(s+1)K'\nabla\rho\cdot\ls\nabla\nabla n\big)\\
&\qquad\qquad\qquad\qquad\qquad\qquad-K\phi_s\Lambda^{s}\nabla n\ls  \nabla\mathrm{div}(\rho v)\,\diff x\\
&=\colon \sum_{j=1}^{3}\int {\Lambda^{s}\partial_j n}\cdot G_{s}^{j}\,\diff x+R,
\end{aligned}
\end{equation}
where
$R$ stands for the terms that can be controlled by
$$R\lesssim (\|(n,\na n,v)\|_{H^s}+\|(\vr,w)\|_{W^{s+2,\infty}})\|(n,\na n,v)\|_{H^s}^2.$$

and
\begin{equation*}
\begin{aligned}
G_{s}^{j}&=-\rho\phi_s K'\na \rho\cdot \big((s+1)\na \ls v_j+\partial_j\ls v\big)
+\partial_j\div(\rho K\phi_s\ls v)
-K\phi_s\p_j\div\ls (\rho v)
\end{aligned}
\end{equation*}
To find the cancellation, we extract the lower order terms as:
\begin{equation*}
\begin{aligned}
G_{s}^{j}
&=-\rho\phi_s K'\na \rho\cdot \big((s+1)\na \ls v_j+\partial_j\ls v\big)+\rho\partial_j(K\phi_s)\ls \div v\\
&\quad+\rho(K\phi_s)'\nabla\rho \cdot\p_j\ls v-K\phi_s[\ls,\rho]\p_j\div v+K\phi_s v\cdot \nabla \Lambda^{s} \partial_j n+l.o.t\\
\end{aligned}
\end{equation*}
Note again that by Lemma \ref{secondcomu},
\begin{equation}
\begin{aligned}
    &[\Lambda^{s},\rho]\mathrm{div}\partial_j v-s\nabla\rho\cdot\Lambda^{s}\partial_j\mathcal{P}^{\perp} v\nonumber\\
    &=[\Lambda^{s},\rho]\mathrm{div}\partial_j v-s\na\rho\cdot\Lambda^{s-2}(1-\Delta)\p_j \mathcal{P}^{\perp} v\nonumber\\
    &=[\Lambda^{s},\rho]\mathrm{div}\partial_j v+s\na\rho\cdot \nabla \Lambda^{s-2}\div\p_j v-s\na\rho \cdot\Lambda^{s-2}\p_j \mathcal{P}^{\perp} v\nonumber\\
    &=[\Lambda^{s},\rho]\mathrm{div}\partial_j v-\f{1}{\mathrm{i}}\{\langle\xi\rangle^s,\rho\}(D)(\mathrm{div}\partial_j v)-s\na\rho \cdot\Lambda^{s-2}\p_j \mathcal{P}^{\perp} v\\
    &=l.o.t
\end{aligned}
\end{equation}
We thus have, by combining the fact $v=\mathcal{P}v+\mathcal{P}^{\perp}v$ and
$\partial_j(\mathcal{P}^{\perp}v)_{l}=\partial_l(\mathcal{P}^{\perp}v)_{j},$ that
\begin{equation}\label{def of G}
    \begin{aligned}
     G_s^j&=   \big((K\phi_s)'\rho-sK\phi_s-(s+2)\rho\phi_s K'\big)\na \rho\cdot\ls\partial_j (\mathcal{P}^{\perp}v)+\rho(K\phi_s)'\partial_j\rho\ls \div \mathcal{P}^{\perp}v\\
     &\quad-(s+1)\rho\phi_s\nabla K\cdot\ls \nabla (\mathcal{P}v)_j+K\phi_s v\cdot \nabla \Lambda^{s} \partial_j n+\rho K\phi_s'\na\rho \cdot\partial_j\ls\mathcal{P}v+l.o.t\\
     &=\rho (K\phi_s)'\big(\partial_j \rho \ls \div(\mathcal{P}^{\perp}v)-\na\rho\cdot\ls(\p_j\mathcal{P}^{\perp}v) \big)\\
     &\quad-(s+1)\rho\phi_s\nabla K\cdot\ls \nabla (\mathcal{P}v)_j+K\phi_s v\cdot \nabla \Lambda^{s} \partial_j n+\rho K\phi_s'\na\rho \cdot\partial_j\ls\mathcal{P}v+l.o.t\\
    \end{aligned}
\end{equation}
 Note that in the second equality, we have used the definition of $\phi_s(\rho)=(\rho K(\rho))^{\frac{s}{2}}$
 which satisfies:
 $$2(K\phi_s)'\rho=sK\phi_s+(s+2)\rho\phi_sK'.$$

We first observe that the contribution of the term \(K\phi_sv\cdot\na \ls \partial_j n\) of \(G_s^j\) in the integral \(I_5\) may be easily handled by integration by parts. Indeed,

\begin{equation}\label{sec2_eq:7}
\begin{aligned}
\int {\Lambda^{s}\nabla n}\cdot (K\phi_s v\cdot \nabla\Lambda^{s} \nabla n)\,\diff x&=-\f{1}{2}\int \mathrm{div}(K\phi_s v)|\Lambda^{s} \nabla n|^2\,\diff x\\
&\lesssim \|v\|_{W^{1,\infty}}(1+\|(\varrho,n)\|_{W^{1,\infty}})\|(\nabla n,v)\|_{H^s}^2.
\end{aligned}
\end{equation}
 Moreover, integrating by parts twice, using the fact that $\partial_j(\mathcal{P}^{\perp}v)_{l}=\partial_l(\mathcal{P}^{\perp}v)_{j}$, one gets
\begin{equation}\label{sec2_eq:9}
\begin{aligned}
&\int \rho (K\phi_s)'{\Lambda^{s}\partial_j n}\big(\partial_j \rho \ls \div(\mathcal{P}^{\perp}v)-\na\rho\cdot\ls(\p_j\mathcal{P}^{\perp}v) \big)\,\diff x\\
&=\int
\ls  (\mathcal{P}^{\perp}v)_{j}\bigg(\partial_l\big((K\phi_s)'\rho\partial_l\rho\big)
{\Lambda^{s}\partial_j n} -
\partial_j \big((K\phi_s)'\rho\partial_l\rho\big){\Lambda^{s}\partial_l n}\bigg)\,\diff x\\
&\lesssim \| (n,\varrho)\|_{W^{2,\infty}}\|(\nabla n,v)\|_{H^{s}}^2.
\end{aligned}
\end{equation}
Similarly, using that $\div(\mathcal{P}v)=0,$ we have by integrating by parts twice
\begin{equation}\label{sec2_eq:7.5}
\begin{aligned}
&\int \rho\phi_s\na K {\Lambda^{s}\partial_j n}\cdot \ls\na(\mathcal{P}v)_j \,\diff x\\
&=\int
\ls(\mathcal{P}v)_{j}\big(\partial_j(\rho\phi_s\nabla K)\ls \nabla n-\div(\rho\phi_s\nabla K)\ls\partial_j n\big)\diff x\\
&\lesssim \|(\vr,n)\|_{W^{1,\infty}}\|(\na n,v)\|_{H^s}^2.
\end{aligned}
\end{equation}

Gathering \eqref{I5}-\eqref{sec2_eq:9}, we achieve that:
\begin{equation}\label{I5final}
I_5=\sum_{j=1}^3\int \rho K\phi_s' \ls\partial_j n \na\rho \cdot \partial_j \ls\mathcal{P}v\diff x+R
\end{equation}
where
$$R\lesssim (\|(n,\na n,v)\|_{H^s}+\|(\vr,w)\|_{W^{s+3,\infty}})\|(n,\na n,v)\|_{H^s}^2.$$

One see that $I_5$ is likely to lose one derivative if $\mathcal{P}v$ is not identical to zero.
To overcome this difficulty, it is necessary to introduce another gauge function to find more cancellations.
Performing $\Lambda^{s}$ on \(\eqref{EK2}_2\), and multiplying it by a function
$\varphi_{s}(\rho)$ (which will be determined later) that is positive on the interval $\bar{\rho}/2\leq\rho\leq 3\bar{\rho}/2$, we get
\begin{equation}\label{sec2_eq:11}
\begin{aligned}
&\partial_t(\varphi_{s}(\rho)\Lambda^{s} v)+g(\bar{\rho})\nabla(\varphi_{s}\Lambda^{s}n)-\nabla[\varphi_s \ls \big(K(\rho)\Delta n+K'(\rho)\nabla n\cdot(\nabla n+2\nabla\tilde{\rho})\big)]\\
&=-\Lambda^{s} (K(\rho)\Delta n)\nabla\varphi_{s}-\varphi_{s}\Lambda^{s}(u\cdot\nabla v)+\partial_t\varphi_{s}\Lambda^{s} v+g(\bar{\rho}) \Lambda^{s} n\nabla\varphi_{s}-\varphi_{s}\Lambda^{s}(v\cdot\nabla w)\\
&\qquad +\Lambda^s (K'(\rho)\nabla n\cdot(\nabla n+2\nabla\tilde{\rho}))\nabla\varphi_s+\varphi_s \Lambda^s F\\
&=-\Lambda^{s} (K(\rho)\Delta n)\nabla\varphi_{s}-\varphi_{s}(u\cdot\nabla \Lambda^{s} v)+l.o.t,
\end{aligned}
\end{equation}
where $l.o.t$ stands for the terms whose $L^2$ norm can be controlled by
\begin{equation}\label{sec2_eq:12}
\begin{aligned}
\|l.o.t\|_{L^2}\lesssim
(\|(\varrho,w)\|_{W^{s+3,\infty}}+\|(n,\nabla n,v)\|_{H^s})\|(n,\nabla n,v)\|_{H^s}.
\end{aligned}
\end{equation}

Multiplying \eqref{sec2_eq:11} by $\mathcal{P}(\varphi_{s}(\rho)\Lambda^{s} v)$ and using integration by parts, one has
\begin{equation}\label{sec2_eq:13}
\begin{aligned}
\frac{1}{2}\frac{\diff}{\diff t} \int|\mathcal{P}(\varphi_{s}(\rho)\Lambda^{s} v)|^2\,\diff x&=- \int {\mathcal{P}(\varphi_{s}(\rho)\Lambda^{s} v)}
\Lambda^{s}(K(\rho)\Delta n)\nabla\varphi_{s}\,\diff x\\
&\quad+\frac{1}{2}\int
\mathrm{div}(\varphi_{s}^2u)|\Lambda^{s}\mathcal{P} v|^2\,\diff x+R\\
&= \sum_{j,l=1}^3\int \partial_{l}\varphi_s \Lambda^{s}(K(\rho)\partial_j n )
{\partial_j\big[\mathcal{P}\left(\varphi_{s}\Lambda^{s} v\right)\big]_{l}}\,\diff x+R\\
&= \sum_{j,l=1}^3\int\partial_{l}\varphi_s \varphi_s K(\rho)\Lambda^{s}\partial_j n
{\partial_j(\mathcal{P} \Lambda^{s} v)_{l}}\,\diff x+R
\end{aligned}
\end{equation}
where $R$ represents the terms that do not lose derivatives, that is
\begin{equation}\label{sec2_eq:14}
\begin{aligned}
|R|\lesssim (\|(\varrho,w)\|_{W^{s+3,\infty}}+\|(n,\nabla n,v)\|_{H^s})\|(n,\nabla n,v)\|_{H^{s}}^2
\end{aligned}
\end{equation}
We mention that we have used the estimate
\begin{equation*}
\begin{aligned}
\|[\partial_j(Id-(\Delta)^{-1}\partial_l\mathrm{div}),f]g\|_{L^2}\lesssim \|g\|_{L^2}\| f\|_{H^s},
\end{aligned}
\end{equation*}
due to \eqref{corcom1}.

We now choose  $\varphi_s(\rho)$ satisfying the condition
$(\varphi_s^2(\rho))'=-2\rho\phi_s'
$ which
cancels the the first terms of  \eqref{I5final} and
\eqref{sec2_eq:13}. More precisely, one can choose
\begin{equation}\label{defofgauge2}
\varphi_s(\rho)=\sqrt{A_s(\rho)}
\end{equation}
where $A_s$ is one primitive of function: $\rho\rightarrow -2\rho\phi_s'(\rho)$ which has positive lower bound on the interval: $\bar{\rho}/2\leq \rho\leq 3\bar{\rho}/2.$
(For some special case, say $K(\rho)\equiv1$, one could write
$\phi_s(\rho)$ explicitly by choosing $\varphi_s(\rho)=\sqrt{-\f{2s}{s+2}\rho^{s+2}+M_s}$
with a constant $M_s=2\cdot(\frac{3}{2}\bar{\rho})^{\frac{s}{2}+1}+1$ which ensures that $\varphi_s(\rho)>1$ uniformly in $s$ on the interval $\bar{\rho}/2\leq \rho(t,x)\leq 3\bar{\rho}/2$).

Gathering the estimates \eqref{I5final}-\eqref{sec2_eq:14}, we find that:
\begin{equation}\label{sec2_eq:16}
\frac{1}{2}\frac{\diff}{\diff t} \int|\mathcal{P}(\varphi_{s}(\rho)\Lambda^{s} v)|^2\,\diff x +I_5
=R
\end{equation}
where ${R}$ stands for those terms that do not lose derivatives, namely, we have that:
$$|{R}|\lesssim (\|(\vr,w)\|_{W^{s+3,\infty}}+\|(n,\na n,v)\|_{H^s})\|(n,\nabla n,v)\|_{H^s}^2.$$

We define the modified energy by
\begin{equation*}
\begin{aligned}
{E}_s(t)=\frac{1}{2}
\int \phi_s(\rho)\big(g(\bar{\rho})|\Lambda^{s}n|^2+K(\rho)|\Lambda^{s}\nabla n|^2+ \rho|\Lambda^{s}v|^2\big)+|\mathcal{P}(\varphi_s\Lambda^{s} v)|^2\,\diff x.
\end{aligned}
\end{equation*}
where $\phi_s(\rho)=(\rho K(\rho))^{\frac{s}{2}}.$
We claim that ${E}_s(t)$ 
is equivalent to $\|(n,\na n,v)\|_{H^s}^2$
as long as \(\|\vr+n\|_{H^s}\)
is sufficiently small. Since
$K(\rho)\geq K_0>0$ for
$\f{1}{2}\bar{\rho}\leq \rho\leq \f{3}{2}\bar{\rho},$
we see that the first three terms are equivalent to $\|(n,\na n,v)\|_{H^s}^2$,
 we thus only need to take care of the last term in \({E}_s(t)\).
Indeed, by the identity
\begin{equation*}
\begin{aligned}
\mathcal{P}(\varphi_s\Lambda^{s} v)=\varphi_s(\mathcal{P} \Lambda^{s} v)+[\mathcal{P},\varphi_s]\Lambda^{s} v,
\end{aligned}
\end{equation*}
and the commutator estimate \eqref{com2},
\begin{equation}
\begin{aligned}
\|[\mathcal{P},\varphi_s]\Lambda^{s} v\|_{L^2}
\lesssim \|\nabla (\varphi_s(\rho))\|_{H^{s-1}}\|v\|_{H^{s-1}}&=\|\nabla (\varphi_s(\rho)-\varphi_s(\bar{\rho}))\|_{H^{s-1}}\|v\|_{H^{s-1}}\\
&\leq C(\|\rho-\bar{\rho}\|_{L^{\infty}})\|\rho-\bar{\rho}\|_{H^s}\|v\|_{H^{s-1}},
\end{aligned}
\end{equation}
for \(s>\f{5}{2}\).
By recalling $\rho-\bar{\rho}=\vr+n$, one easily see that,  as long as $\|\vr+n\|_{H^s}$ is sufficiently small, it holds that
\begin{equation*}
\begin{aligned}
\|\mathcal{P}(\varphi_s\Lambda^{s} v)\|_{L^2}^2\approx  \|\varphi_s(\mathcal{P} \Lambda^{s} v)\|_{L^2}^2\approx \|\mathcal{P} \Lambda^{s} v\|_{L^2}^2.
\end{aligned}
\end{equation*}
Hence
\begin{align}\label{equivalent identity}
\frac{1}{C_0}\|(n,\nabla n,v)\|_{H^s}^2\leq E_s(t)\leq C_0 \|(n,\nabla n,v)\|_{H^s}^2,
\end{align}
for some positive constant \(C_0\).

We conclude from \eqref{sec2_eq:2}-\eqref{sec2_eq:4} and \eqref{sec2_eq:16} that
\begin{equation*}
\begin{aligned}
\frac{\diff}{\diff t}{E_s}(t)&\lesssim \|(n,\nabla n,v)\|_{H^s}^{3/2} +\|(\varrho,w)\|_{W^{s+3,\infty}}\|(n,\nabla n,v)\|_{H^s}\\
&\leq C_1 \big(E_s^{3/2}(t)+\|(\varrho,w)\|_{W^{s+3,\infty}}E_s(t)\big),
\end{aligned}
\end{equation*}
for some positive constant \(C_1\).
Moreover, by Theorem \ref{thmek1}, one has that:
$$\|(\varrho,w)(t)\|_{W^{s+3,\infty}}\leq C\delta_1(1+t)^{-\alpha}$$
which leads to:
\begin{equation}\label{energyineq}
\begin{aligned}
\frac{\diff}{\diff t}{E_s}(t)&\leq
 C_1 \big(E_s^{3/2}(t)+C\delta_1 (1+t)^{-\alpha}E_s(t)\big)
\end{aligned}
\end{equation}

\noindent{{\bf{Conclusion for Theorem \ref{thmek1}.}}
Theorem \ref{thmek1} stems from a standard continuity argument. We define the maximal existence time $T_{\star}$ by
\begin{equation}
\begin{aligned}
T_{\star}=\sup\{T\big|
(n,\nabla n,u)\in C({[0,T],H^s}):\ E_s(T)\leq 2MC_0 \epsilon_{1}^2 \},
\end{aligned}
\end{equation}
where $M=e^{\frac{CC_1}{\alpha-1}\delta_1}$, $C,\alpha$ are constants that appear in the statement of Theorem \ref{thmek1}.
In view of \eqref{equivalent identity},
as long as $\epsilon_1$ is sufficitenly small, we have: $\|n(t)\|_{L^{\infty}}\leq \bar{\rho}/2$ for any $0\leq t\leq T_{*}$,
which, combined with \eqref{bound for varrho}, yields $\bar{\rho}/2\leq \rho\leq 3\bar{\rho}/2$ for $0\leq t\leq T_{*}.$

Define further $T_0=\min\{T_{\star},\kappa\epsilon_{1}^{-1}\}$ with $\kappa$ being small to be chosen later.
By  \eqref{energyineq}, one easily gets by Gr\"onwall's inequality for $t\leq T_{0}\leq \kappa\epsilon_{1}^{-1}$ that
\begin{equation*}
\begin{aligned}
E_s(t)& \leq \exp\big(\int_0^t C_1 C\delta_1 (1+\tau)^{-\alpha}\diff \tau\big)\bigg( E_s(0)+C_1 \int_0^t E_s^{3/2}(\tau)\,\diff \tau\bigg)\\
&\leq M\bigg( E_s(0)+C_1 \int_0^t E_s^{3/2}(\tau)\,\diff \tau\bigg)\\
&\leq M\left(C_0\epsilon_{1}^2+C_1 T_0 (2KC_0\epsilon_{1}^2)^{\frac{3}{2}}\right)
\leq \frac{3}{2}KC_0\epsilon_{1}^2,
\end{aligned}
\end{equation*}
by choosing $\kappa=\big(2C_1C_0^{\frac{1}{2}}(2K)^{\frac{2}{3}}\big)^{-1}$, which leads to, by combining the local existence theory, $T_0=\kappa\epsilon_{1}^{-1}<T_{\star}$.

{\bf Proof of the second part of Theorem \ref{thmek0}}
To finish the proof of Theorem \ref{thmek0}, we suppose further that $\mathcal{P}u_0\in H^{N}$, we prove briefly that the solution  belongs to \eqref{EK} satisfies:
$$(\rho-\bar{\rho},u) \in C([0,T_{\epsilon_1}],H^{N+1}(\mathbb{R}^3)\times H^{N}(\mathbb{R}^3)).$$
Let us define functional:
\begin{equation*}
\begin{aligned}
\mathscr{E}_N(t)=\frac{1}{2}
\int \phi_N(\rho)\big(g(\bar{\rho})
|\Lambda^{N}(\rho-\bar{\rho})|^2+ K(\rho)
|\Lambda^{N}\nabla \rho|^2 +\rho|\Lambda^{N}v|^2\big)+|\mathcal{P}(\varphi_N\Lambda^{N} v)|^2\,\diff x.
\end{aligned}
\end{equation*}
where
$\phi_N=(\rho K)^{\frac{N}{2}}$, $\varphi_N(\rho)$ is defined in
\eqref{defofgauge2}. 
By calculations similar to that in the proof of Theorem \ref{thmek2}, one could prove that:
\begin{equation}\label{energyineqpart2}
 \frac{\diff}{\diff t}\mathscr{E}_N(t)   \lesssim \|\rho-\bar{\rho}\|_{W^{2,\infty}\cap H^s}\mathscr{E}_N(t).
\end{equation}
Note that we will always use tame estimate \eqref{appendix:2.75} and commutator estimate \eqref{appendix:4},\eqref{appendix:7} for the product terms and commutator terms in the expression of $\frac{\diff}{\diff t}\mathscr{E}_N(t).$

In light of the fact $\rho-\bar{\rho}\in C([0,T_{\epsilon_1}],H^{s+1}(\mathbb{R}^3))$
($s>5/2$), energy inequality \eqref{energyineqpart2}
 and Gr\"onwall inequality, we have for any $0\leq t\leq T_{\epsilon_1}$
$$\mathscr{E}_N(t)   \lesssim e^{ct}\mathscr{E}_N(0)$$
for some constant $c.$

\begin{section}{Large time existence of two-fluid Euler-Maxwell equation}
To show the versatility of the approach proposed in the Introduction, we will prove the similar results analogue to Theorem \ref{thmek0}
for 'two-fluid' Euler-Maxwell system.
As a physical model to describe the  dynamics of plasma, namely electrons and ions, 'two-fluid' Euler-Maxwell system reads:
\begin{equation} \label{EM0}
 \left\{
\begin{array}{l}
\displaystyle \pt n+\mathrm{div}\big((1+n)v\big)=0,\\
\displaystyle \iota(\pt v+  v\cdot\nabla v)+d\na n+E+v\times B=0, \\
\displaystyle
\pt \vr +\div\big( (1+\vr) u\big)=0,\\
\displaystyle \pt u+  u\cdot\na u +\na\vr-E-u\times B=0
, \\
\pt B +\na \times E=0,\\
\pt E-\f{1}{\iota}\na\times B=[(n+1)v-(\vr+1)u]\\
\displaystyle \div B=0, \div E=\vr-n\\
\displaystyle (n,v,\vr,u,E,B)|_{t=0}=(n^0,v^0,\vr^0,u^0,E^0,B^0).
\end{array}
\right.
\end{equation}
where $1+n,v$(resp. $1+\vr,u$) stand for the density and velocity of  electrons (resp. ions), $E$, $B$ stand for the electric and magnetic field, 
$\iota, d$ are two constant parameters. Note that we have chosen the reference state being $(1,0,1,0,0,0)$.
Before going further, we shall point out some observations. At first, from the equation for $B,E$, $\div B(t)=0,\div E(t)=\vr(t)-n(t)$ for any $t>0$, as long as $\div B^0=0, \div E^{0}=\vr^0-n^0.$
Secondly, as indicated in \cite{MR3742601} there are two important 
generalized vorticities $Y(t)=: B-\frac{1}{\iota}\na\times v$ and $W(t)=: B+\na\times u$ which satisfy the
evolution equations:
\begin{eqnarray}\label{genvor}
\pt Y=\na\times(v\times Y)\qquad
\pt W=\na\times(v\times W)
\end{eqnarray}
Note that by direct energy estimate, if $Y^0=B^0-\f{1}{\iota}\na\times v^0=0$ and $W^0=B^0+\na\times u^0=0$, then this property will propagate, and we call this kind of flow as 'generalized irrotational flow'.

 Unlike the compressible Euler equation for which singularity formation happens for finite time \cite{MR3360663},\cite{MR815196}
global small smooth solutions of \eqref{EM0} has been constructed in \cite{MR3450481}
with generalized irrotational initial datum by using 'space-time resonance' technique and delicate Fourier analysis.
\begin{theorem}
[Theorem 1.1 of \cite{MR3450481}  ]\label{thmem1}
There  exists $\delta_0>0$ small, $N_0$ large enough, $\beta>1$, if the initial data satisfy the following:
$$\|(n^0,v^0,\vr^0,u^0,E^0,B^0)\|_{H^{N_0}\cap Z}\leq \delta_0,$$
$$\div E^0-\vr^0+n^0=0, \qquad B^0-\f{1}{\iota}\na\times v^0=B^0+\na\times u^0=0,$$
then the system \eqref{EM0} has a unique global solution in $C([0,\infty),H^{N_0})$ and satisfies the
following:
\begin{equation}
\|(n,v,\vr,u,E,B)\|_{H^{N_0}}+(1+t)^{\beta}\|(n,v,\vr,u,E,B)\|_{W^{4,\infty}}\lesssim \delta_0.
\end{equation}
\end{theorem}

\begin{remark}
In the statement of last Theorem, $Z$ stands for the norm involving localization of both space and frequency which is compatible to the fractional weighted norm $\|x^{1^{+}}f\|_{L^2}$. For precise definition, one could refer to definition 4.1 of \cite{MR3450481}.
\end{remark}
\begin{remark}
In contrast with the global existence for \textbf{small} irrotational solutions, the formation of singularity is likely to happen for large initial data. In \cite{MR3841988}, the blow up result is obtained for
spherically symmetric solutions to 'one-fluid' Euler-Poisson equations, which is a special case of two-fluid' Euler-Maxwell equations by neglecting the motion of ions and the effects of magnetic field. The method proposed by them seems could be extend to the general case.

\end{remark}
In the following, we aim to study the long time existence  of solution of \eqref{EM0} with general data,
that is $B^0-\f{1}{\iota}\na\times v^0 \neq
 0$ or $B^0+\na\times u^0\neq 0$. 
 We have by
 \eqref{genvor} and identity:
 $\na\times(v\times Y)=Y\cdot \na v-v\cdot\na Y-Y\div v$ that
 \begin{equation}\label{Y}
 \pt Y=Y\cdot \na v-v\cdot\na Y-Y\div v, \quad \pt W=W\cdot \na u-u\cdot\na W-W\div u
 \end{equation}
for which there is no any dispersive or dissipation structure that can be used. It seems that one can only expect the lifespan of $Y, W$ in some Sobolev space (say $H^s, s>\frac{5}{2}$) is proportional to $1/\|Y^0\|_{H^s}$.

 The following is the main result concerning to 'two-fluid' Euler-Maxwell equations:
\begin{theorem}\label{thm for EM}
  There exists three constants $\delta, \epsilon>0$ small, $\f{5}{2}<s\leq3$ and $N_0$ large. If the initial datum $(n^0,v^0,\vr^0,u^0,E^0,B^0)$
  satisfies the following:
  \begin{eqnarray*}
  \div B^0=0, \qquad \div E^0+n^0-\vr^0=0,\\
  \|(n^0,\mathcal{P}^{\perp}v^0+\f{1}{\iota}(\Delta)^{-1}\na\times B^0,\vr^0,\mathcal{P}^{\perp}u^0-(\Delta)^{-1}\nabla\times B^0, E^0,B^0)\|_{H^{N_0}\cap Z}< \delta,\\
    \|(0,\mathcal{P} v^0-\f{1}{\iota}(\Delta)^{-1}\nabla\times B^{0},0,\mathcal{P} u^0+(\Delta)^{-1}\nabla\times B^{0},0)\|_{H^s}<\epsilon.
 \end{eqnarray*}
 Then the Euler-Maxwell equation \eqref{EM0}
 admits a solution in $C([0,T_{\epsilon}],H^s)$
  with $T_{\epsilon}\gtrsim \epsilon^{-1}$.

In addition if
$\mathcal{P} v^0-\f{1}{\iota}(\Delta)^{-1}\nabla\times B^{0},\mathcal{P} u^0+(\Delta)^{-1}\nabla\times B^{0}$ belongs to $H^{N_0}$,
then 
the solution lies in
$C([0,T_{\epsilon}],H^{N_0})$
with an exponential growth:
$$\|(n,v,\vr,u,E,B)(t)\|_{H^{N_0}}\lesssim e^{ct}\|(n,v,\vr,u,E,B)(0)\|_{H^{N_0}}.$$
\end{theorem}
 \begin{proof}[Proof of Theorem \ref{thm for EM}]
As explained in the introduction, we split the system into two systems. More precisely, we write
 $$(n,v,\vr,u,E,B)=(n_1,v_1,\vr_1,u_1,E_1,B_1)+(n_2,v_2,\vr_2,u_2,E_2,B_2),$$
where $(n_1,v_1,\vr_1,u_1,E_1,B_1)$ is the global solution  of system \eqref{EM0} provided by Theorem \ref{thmem1} with initial data
$(n^0,\mathcal{P}^{\perp}v^0+\f{1}{\iota}(\Delta)^{-1}\na\times B^0,\vr^0,\mathcal{P}^{\perp}u^0-(\Delta)^{-1}\nabla\times B^0, E^0,B^0)$. Then $(n_2,v_2,\vr_2,u_2,E_2,B_2)$ solves the equations
which is a perturbation of the original system  by that of
$(n_1,v_1,\vr_1,u_1,E_1,B_1)$. 
That is:
\begin{equation} \label{EM2}
 \left\{
\begin{array}{l}
\displaystyle \pt n_2+\div((1+n)v_2+n_2v_1)=0,\\
\displaystyle \iota(\pt v_2+  v\cdot\na v_2+v_2\cdot\na v_1)+d\na n_2+E_2+v\times B_2+v_2\times B_1=0, \\
\displaystyle
\pt \vr_2 +\div( (1+\vr) u_2+\vr_2 u_1)=0,\\
\displaystyle \pt u_2+  u\cdot\na u_2+u_2\cdot\na u_1 +\na\vr_2-E_2-u\times B_2+u_2\times B_1=0
, \\
\pt B_2 +\na \times E_2=0,\\
\pt E_2-\f{1}{\iota}\na\times B_2=[(n+1)v_2+n_1 v_2-(\vr+1)u_2-\vr_1u_2]\\
\displaystyle \div B_2=0, \div E_2=\vr_2-n_2\\
\displaystyle (n_2,v_2,\vr_2,u_2,E_2,B_2)|_{t=0}=(0,\mathcal{P} v^0-\f{1}{\iota}(\Delta)^{-1}\na\times B^0,0,\mathcal{P} u^0+ (\Delta)^{-1}\nabla\times B^0,0,0).
\end{array}
\right.
\end{equation}

 We shall then prove Theorem \ref{thm for EM} by direct energy estimate. Define the energy functional:
 \begin{eqnarray*}
 \mathcal{E}_{s}=\f{1}{2}\big(\int d|\ls n_2|^2+ \iota (1+n)|\lk v_2|^2\diff x+\int |\ls \vr_2|+(1+\vr)|\ls u|^2\diff x\\
 +\int |\ls E|^2+\f{1}{\iota}|\ls B|^2\diff x \big)
 \end{eqnarray*}
 where $n=n_1+n_2,\vr=\vr_2+\vr_2$.

 Taking the time derivative of $\cE_{s}$ and using the equations \eqref{EM2}, we easily get:
 \begin{eqnarray*}
 \pt \mathcal{E}_s
 &=&- \int d{\ls n_2}\bigg(\ls\div\big((1+n)v_2\big)-\div\left ((1+n)\ls v_2\right)\bigg)\\
 &&\qquad \qquad\qquad\qquad +{\ls \vr_2}\bigg(\ls\div\left((1+\vr)u_2\right)-\div\left ((1+\vr)\ls u_2\right)\bigg)
 \diff x\\
 &&- \int d {\ls n_2}\ls \div(n_2 v_1)+{\ls \vr_2}\ls \div(\vr_2 u_1)\diff x\\
 &&-\bigg( \int \iota (1+n){\ls v_2}(v\cdot\na \ls v_2)+(1+\vr){\ls u_2}(u\cdot\na \ls u_2)\diff x\\
 &&\qquad\qquad\qquad\qquad\qquad\qquad\qquad\qquad\qquad-\int \iota\pt n|\ls v_2|^2+\pt\vr|\ls u_2|^2\diff x\bigg)
 \\
 &&- \int\iota (1+n){\ls v_2}\bigg(  [\ls, v_2]\nabla v_2+\ls (v_2\cdot \na v_1)\bigg)\\
 &&\qquad\qquad \qquad\qquad\qquad\qquad+(1+\vr){\ls u_2}\bigg(  [\ls, u_2]\na u_2+\ls (u_2\cdot \na u_1) \bigg)\diff x\\
 &&-\int(1+n){\ls v_2}\ls(v\times B_2+v_2\times B_1)+(1+\vr){\ls u_2}\ls(u\times B_2+u_2\times B_1)\diff x\\
 &&+ \int {\ls E_2}\bigg([\ls, n]v_2-[\ls, \vr]u_2+n_2v_1-\vr_2u_1\bigg)\diff x\\
 &=:& J_1+J_2+\cdots J_6
 \end{eqnarray*}

 We now estimate $J_1$-$J_6$ rigorously.
 For $J_1$, we write:
 $$\ls\div\big((1+n)v_2\big)-\div\left ((1+n)\ls v_2\right)=[\ls, n]\div v_2+v_2\cdot\ls \na n+[\ls,\na n,v_2],$$
 $$\ls\div\big((1+\vr)u_2\big)-\div\left ((1+\vr)\ls u_2\right)=[\ls, \vr]\div u_2+u_2\cdot\ls \na \vr+[\ls,\na \vr,u_2],$$
 where we denote for $k\geq 1$, $[\ls, f,g]=\ls(fg)-g\ls f-f\ls g.$

For $J_1$, we use integration by parts and   commutator estimates \eqref{appendix:3}-\eqref{appendix:4} to get:
 \begin{eqnarray}\label{J1}
 J_1\lesssim \big(\|\na (n_2,v_2,\vr_2,u_2)\|_{L^{\infty}}
 +\|(n_1,v_1,\vr_1,u_1)\|_{W^{s+1,\infty}}\big)\|(n,v,\vr,u,E,B)\|_{H^s}^2.
 \end{eqnarray}

 For example, by Lemma \ref{classical com}, we have:
 \begin{equation*}
 \|[\ls,n_1]\div v_2\|_{L^2}\lesssim \| n_1\|_{W^{s,\infty}}\|\na v_2\|_{H^{s-1}}.
 \end{equation*}

  $J_2,$ could be controlled in the same manner:
 \begin{eqnarray}
J_2&=&
 -\Re \int d \overline{\ls n_2}\big(v_1\cdot\na\ls n_2+[\ls, v_1]\na n_2+\ls (n_2\div v_1)\big)\nonumber\\
 &&\qquad \qquad \qquad+ \overline{\ls \vr_2}\big(u_1\cdot\na \ls\vr_2+[\ls, u_1]\na \vr_2+\ls (\vr_2\div u_1)\big)\diff x\nonumber\\
 &\lesssim&\|n_2\|_{H^{k}}^2 \|\na v_1\|_{W^{s,\infty}}+\|\vr_2\|_{H^s}^2\|\na u_1\|_{W^{s,\infty}}.
 \end{eqnarray}
Next, from the equation satisfied by $n,\vr$, $(\eqref{EM0}_1,\eqref{EM0}_3)$: $\pt n+\div((1+n)v)=0,\pt\vr+\div((1+\vr)u)=0$ we have $J_3=0$.

 For $J_4$, by commutator estimate \eqref{appendix:3}-\eqref{appendix:4} again, we have that:
 \begin{eqnarray}
 J_4&\lesssim& (1+\|n\|_{L^{\infty}})\|\ls v_2\|_{L^2}(\|\na v_2\|_{L^{\infty}}\|v_2\|_{{H}^{s}}+\|\na v_1\|_{W^{s,\infty}}\|v_2\|_{H^s})\nonumber\\
 &&\qquad +(1+\|\vr\|_{L^{\infty}})\|\ls u_2\|_{L^2}(\|\na u_2\|_{L^{\infty}}\|u_2\|_{{H}^{s}}+\|\na u_1\|_{W^{s,\infty}}\|u_2\|_{H^s})\nonumber\\
 &\lesssim& \big(\|\na(v_2,u_2)\|_{L^{\infty}}+\|(v_1,u_1)\|_{W^{s+1,\infty}}\big)\|(v_2,u_2)\|_{H^s}^2
 \end{eqnarray}

 Similarly:
 \begin{eqnarray}
 J_5\lesssim \big(\|(v_2,u_2,b_2)\|_{L^{\infty}}+\|(v_1,u_1,B_1)\|_{W^{s,\infty}}\big)\|(v_2,u_2,B_2)\|_{H^{s}}^2.
 \end{eqnarray}
 \begin{eqnarray}\label{J6}
 J_6\lesssim &&\big(\|\na(n_2,v_2,\vr_2,u_2)\|_{L^{\infty}}+\|(n_1,v_1,\vr_1,u_1)\|_{W^{s,\infty}}\big)\nonumber\\
&&\qquad\qquad\qquad\qquad \|(E_2,n_2,v_2,\vr_2,u_2)\|_{H^{s}}^2.
 \end{eqnarray}

 We thus get by collecting the above estimates \eqref{J1}-\eqref{J6}
 \begin{eqnarray*}
 \cE_s &\lesssim&
 \|(n_2,v_2,\vr_2,u_2,E_2,B_2)\|_{H^s}^3\\
 &&+\|(n_1,v_1,\vr_1,u_1,E_1,B_1)\|_{W^{s+1,\infty}}\|(n_2,v_2,\vr_2,u_2,E_2,B_2)\|_{H^s}^2.
 \end{eqnarray*}
 Gronwall's inequality  and continuation arguments then give the lower bound of lifespan $T_{\epsilon}\geq \epsilon^{-1}$. Since it is similar to the case of Euler-Korteweg, we omit the details.

 \end{proof}

\end{section}

\section{Appendix}

In this appendix, we recall and prove some basic commutator estimates frequently used in showing Theorem \ref{thmek0}, based on a Littlewood-Paley decomposition.
Let $\psi\in[0,1]$ be a cut-off function satisfying
$\psi\equiv 1$ on $B(0,3/2)$ and $\psi\equiv 0$ on $B(0,2)^{c}$.
Set $\phi_j(x)=\phi(2^{-j}x)$, with $\phi(x)=\psi(x)- \psi(2x)$ which is supported on the annulus $\{\frac{3}{4}\leq|x|\leq 2\}$. Then it holds that
 \begin{equation*}
 \begin{aligned}
1=\psi(x)+\sum_{j\geq 1}\phi_{j}(x),\quad \text{for\ all}\ x\in \mathbb{R}^3.
\end{aligned}
\end{equation*}
We also recall the homogeneous dyadic block \(\dot\Delta_k\) defined by
\begin{equation*}
\begin{aligned}
\dot\Delta_kf= \mathcal{F}^{-1}\big(\phi_{k}(\xi)\hat{f}(\xi)\big),\quad \text{for\ all}\ k\in \mathbb{Z},
\end{aligned}
\end{equation*}
and homogeneous low-frequency cut-off operator \(\dot S_l\) given by
\begin{equation*}
\begin{aligned}
\dot S_l=\sum_{k\leq l-1}\dot\Delta_k,\quad \text{for\ all}\ l\in \mathbb{Z}.
\end{aligned}
\end{equation*}
Similarly, the nonhomogeneous dyadic block \(\Delta_k\) and homogeneous low-frequency cut-off operator \(S_l\) are defined respectively by
\begin{equation*}
\begin{aligned}
\Delta_{-1}f= \mathcal{F}^{-1}\big(\psi(\xi)\hat{f}(\xi)\big),\ \Delta_kf= \mathcal{F}^{-1}\big(\phi_{k}(\xi)\hat{f}(\xi)\big),\quad \text{for\ all}\ k\in \mathbb{N}.
\end{aligned}
\end{equation*}
and
\begin{equation*}
\begin{aligned}
\ S_l=\sum_{-1\leq k\leq l-1}\Delta_k,\quad \text{for\ all}\ l\in \mathbb{Z}.
\end{aligned}
\end{equation*}

Let $\mathcal{R}$ be the Riesz potential, and $\Lambda^s$
($s\geq 0$) be the Fourier multiplier with symbol $(1+|\xi|)^\f{s}{2}$. We first prove the following:

\begin{lemma}\label{lemcom1}
	We have for $s>\f{3}{2},$
\begin{align}\label{com2}
\|[\mathcal{R}, f]\Lambda^s g\|_{L^2}\lesssim \|
 \na f\|_{H^{s}}\|g\|_{{H}^{s-1}}.
\end{align}
\end{lemma}

\begin{proof} 
	We use decomposition:
	 \begin{align*}
	fg=\sum_{j\in \mathbb{Z}}\dot S_{j-1}f\dot{\Delta}_{j}g+\dot{\Delta}_{j}f \dot S_j g= \dot{T}_{f}g+\tilde{\dot{T}}_{g}f
	\end{align*}
	 to rewrite $[\mathcal{R}, f]\Lambda^s g$ as:
	 \begin{align*}
	[\mathcal{R}, f]\Lambda^s g=[\mathcal{R}, \dot{T}_f]\Lambda^s g+\mathcal{R} (\tilde{\dot{T}}_{\Lambda^s g}f)-\tilde{\dot{T}}_{\mathcal{R} \Lambda^s g}f.
	\end{align*}
	
	The last two terms can be easily treated by Bernstein's inequality:
\begin{equation}\label{appendix:1}
\begin{aligned}
	\|\mathcal{R} (\tilde{\dot{T}}_{\Lambda^s g}f)\|_{L^2}+\|\tilde{\dot{T}}_{\mathcal{R} \Lambda^s g}f\|_{L^2}
	&\lesssim \sum_j \|\dot S_j \Lambda^{s-1}g\|_{L^2}\|\dot{\Delta}_{j}
	\nabla f\|_{L^{\infty}}\\
	&\lesssim \|g\|_{H^{s-1}}\|
	\nabla f\|_{\dot{B}_{\infty,1}^{0}}.
\end{aligned}
\end{equation}
We next handle the first term. Since the frequency of  $\dot{T}_{f}\Lambda^s g$ localizes on the annulus, there exits a $C_{c}^{\infty}$ function $\tilde{\phi}$ supported away from origin such that
$$
	[\mathcal{R}, \dot{T}_f]\Lambda^s g
	=\sum_j \tilde{\phi}(2^{-j}D)\big(\dot S_{j-1}f\dot{\Delta}_{j}\Lambda^s g\big)-\dot S_{j-1}f \big(\tilde{\phi}(2^{-j}D)\dot{\Delta}_{j}\Lambda^s g\big)\\
	=\colon\sum_j A_j.
$$
For $j\leq 0,$  taking advantage of the Bernstein inequality, it is direct to see that:
\begin{equation*}
 \|A_j\|_{L^2}\lesssim \|\dot{S}_{j-1}f\|_{L^{\infty}}\|\dot{\Delta}_j g\|_{L^2}\lesssim \|\nabla f\|_{H^{s-1}} \|\dot{\Delta}_j g\|_{L^2}
\end{equation*}

Now for $j\geq 1,$	denote $\tilde{\chi}_j=\mathcal{F}^{-1}(\tilde{\phi}(2^{-j}\xi)(1+|\xi|^2)^\f{s}{2})$, one may write
\begin{equation*}
\begin{aligned}
	{A_j}(x)&=
	\int \tilde{\chi}_j( y)\dot{\Delta}_{j}g(x-y)(\dot S_{j-1}f(x)-\dot S_{j-1}(x-y))\,\diff y\\
	&=\int \tilde{\chi}_j( y)\dot{\Delta}_{j} g(x-y)\int_{0}^{1}y\cdot \nabla \dot S_{j-1}f(x-y+\tau y)\,\diff \tau\diff y.
\end{aligned}
\end{equation*}
Hence
	\begin{equation*}
	\begin{aligned}
	\|A_j\|_{L^2}\lesssim \|\nabla \dot S_{j-1}f\|_{L^{\infty}}\|\dot{\Delta}_{j}g\|_{L^2} \||\cdot|\tilde{\chi}_j\|_{L^1}\lesssim \|\na f\|_{H^s}\|\dot{\Delta}_{j}g\|_{L^2}2^{j(s-1)}.
	\end{aligned}
	\end{equation*}
Taking $l_j^{2}$ norm of sequence $(A_j)$, one has
	\begin{equation}\label{appendix:2}
	\begin{aligned}
	\|[\mathcal{R}, T_f]\Lambda^s g\|_{L^2}\lesssim
	\|\nabla f\|_{H^s} \|g\|_{H^{s-1}}.
	\end{aligned}
	\end{equation}
	
The desired result \eqref{com2} follows from \eqref{appendix:1} and \eqref{appendix:2}.	
\end{proof}

\begin{corollary}
For any $s>\f{3}{2}$, one has also the commutator estimate:
	\begin{equation}\label{corcom1}
	\begin{aligned}
	\|[\nabla\mathcal{R},f]g\|_{L^2}\lesssim \| \na f\|_{H^{s}}\|g\|_{L^2}.
	\end{aligned}
	\end{equation}
\end{corollary}	
	\begin{proof}
We have the following estimates similar to \eqref{com2}:
	\begin{align}
		\|[\mathcal{R}, f]\na  g\|_{L^2}\lesssim \|
	\nabla f\|_{H^s}\|g\|_{L^2}\nonumber
	\end{align}
	
Then \eqref{corcom1} is the consequence of this estimate and the identity:
	$$
	 [\nabla\mathcal{R},f]g=\mathcal{R}(g\nabla f)+[\mathcal{R},f]\nabla g.
	$$
	\end{proof}	

We will use also the following commutator and
product estimates whose proof are standard and thus omitted. One could refer to \cite{MR2768550} for example.
\begin{lemma}\label{classical com}
	Let $s\geq 1$, we have
		\begin{equation}\label{appendix:2.5}
	\begin{aligned}
	\|\Lambda^s(fg)\|_{L^2}\lesssim \|f\|_{W^{s,\infty}}\| g\|_{H^s},
	\end{aligned}
	\end{equation}
			\begin{equation}\label{appendix:2.75}
	\begin{aligned}
	\|\Lambda^s(fg)\|_{L^2}\lesssim \|f\|_{L^{\infty}}\| g\|_{H^s}+\|g\|_{L^{\infty}}\|f\|_{H^s},
	\end{aligned}
	\end{equation}
	\begin{equation}\label{appendix:3}	\begin{aligned}
 \|[\Lambda^s,f]g\|_{L^2}\lesssim \|g\|_{H^{s-1}}\| f\|_{W^{s,\infty}},
	\end{aligned}
	\end{equation}
\begin{equation}\label{appendix:4}
\begin{aligned}
	\|[\Lambda^s,f]g\|_{L^2}\leq \|\nabla f\|_{L^{\infty}}\|g\|_{H^{s-1}}+\|g\|_{L^{\infty}}\|f\|_{H^{s}}.
\end{aligned}
\end{equation}

\end{lemma}

Given two functions $a(x,\xi),b(x,\xi)$, the Poisson brackets reads
\begin{equation*}
\begin{aligned}
 \{a,b\}=\partial_{\xi}a\cdot\partial_x b-\partial_{\xi}b \cdot\partial_x a.
 \end{aligned}
 \end{equation*}

\begin{lemma}\label{secondcomu}
Let $2\leq s\leq \frac{7}{2}$, we have
\begin{equation}\label{appendix:7}
\begin{aligned}
\|[\Lambda^s,f]g-\mathrm{i}^{-1}\{\langle \xi\rangle^s,f\}(D)g\|_{L^2}
\lesssim \|\na ^2 f\|_{L^{\infty}\cap H^{\frac{3}{2}}}\|g\|_{H^{s-2}},
\end{aligned}
\end{equation}
and
\begin{equation}\label{appendix:8}
\begin{aligned}
\|[\Lambda^s,f]g-\mathrm{i}^{-1}\{\langle\xi\rangle^s,f\}(D)g\|_{L^2}\lesssim \| f\|_{W^{{s+\epsilon},\infty} }\|g\|_{H^{s-2}}.
\end{aligned}
\end{equation}
\end{lemma}


\begin{proof}
	One can refer to  \cite[Lemma A.3]{MR2354691} for the proof of \eqref{appendix:7}. We only sketch the proof of \eqref{appendix:8}. 
	
	We denote by $\partial_k,\partial^k$ the space derivative and frequency derivative respectively. We use decomposition:
\begin{eqnarray*}
[\Lambda^s,f]g-\mathrm{i}^{-1}\{\langle\xi\rangle^s,f\}(D)g&&=\underbrace{[\Lambda^s,{T}_f]g+\mathrm{i}
		T_{\partial_k f}(\partial^{k}\Lambda^{s})(D) g}_{G_1}\\
	&&\quad+\underbrace{\Lambda^s(\tilde{{T}}_{g}f)-\tilde{{T}}_{\Lambda^{s} g}( f)+s\tilde{{T}}_{\Lambda^{s-2} \partial_k g}(\partial_k f)}_{G_2}
\end{eqnarray*}
	
	Taking $\tilde{\phi}$ (defined in the proof of Lemma \ref{lemcom1}), noticing that $\tilde{\phi}_j \equiv 1$ on the support of $\phi_j$,
	we may decompose $G_1$ as
	\begin{equation*}
	\begin{aligned}
	G_1&=
	\sum_{j\in \mathbb{Z}}\bigg(\Lambda^s \tilde{\phi}(2^{-j}D)\big(S_{j-1}f{\Delta}_{j} g\big)\\
	&\quad-S_{j-1}f \big(\tilde{\phi}(2^{-j}D)\Lambda^s{\Delta}_{j} g\big)+\mathrm{i}S_{j-1}\partial_k f \big(\partial^{k}(\tilde{\phi}(2^{-j}\cdot)\langle \cdot\rangle ^s)(D){\Delta}_{j} g\big)\bigg)\\
	&=\colon\sum_{j\in \mathbb{Z}}A^j.
	\end{aligned}
	\end{equation*}
	where we denote $\langle\cdot\rangle=(1+|\cdot|^2)^{\f{1}{2}}.$
	For $j=0$, it is easy to see that:
	\begin{equation*}
	   \|A_0\|_{L^2}\lesssim \|f\|_{L^{\infty}}\|\Delta_0 g\|_{L^2}.
	\end{equation*}
	
For $j\geq 1$,
	denote $\tilde{\chi}_j=\mathcal{F}^{-1}(\tilde{\phi}(2^{-j}\cdot)\langle\cdot\rangle^s)$. By Taylor expansion, one has
\begin{equation*}
\begin{aligned}
	A^j(x)&=\int \tilde{\chi}_{j}( y){\Delta}_j g(x-y)\big(S_{j-1}f(x-y)-S_{j-1}f(x)+\partial_k S_{j-1}f(x)y_k\big)\,\diff y\\
	&=\int \tilde{\chi}_{j}( y){\Delta}_jg(x-y)
	\int_{0}^{1}y^{T}\cdot D^2 S_{j-1}f(x-t y)\cdot y (1-t)\,\diff t \diff y,
\end{aligned}
\end{equation*}
	which yields, for $j\geq 1$
	\begin{equation*}
	\begin{aligned}
	\|A^j\|_{L^2}&\lesssim \||\cdot|^2\tilde{\chi}_j\|_{L^1}\|{\Delta}_j g\|_{L^2}\|D^2 {S}_{j-1}f\|_{L^{\infty}}\\
	&\lesssim 2^{j(s-2)} \|{\Delta}_j g\|_{L^2}\|D^2 f\|_{L^{\infty}}.
	\end{aligned}
	\end{equation*}
Taking $l_{j}^2$ norm of of sequence $(A^j)$, one obtains
	\begin{equation}\label{appendix:9}
	\begin{aligned}
	\|G_1\|_{L^2}\lesssim\|g\|_{H^{s-2}}\|f\|_{W^{s+\epsilon,\infty}}.
	\end{aligned}
	\end{equation}	
As for $G_2$, it can be estimated easily
\begin{equation}\label{appendix:10}
\begin{aligned}
	\|G_2\|_{L^2}&\lesssim \sum_{j\geq -1}2^{js}\|S_{j}g\|_{L^2}\|\Delta_j f\|_{L^{\infty}}\\
	&\lesssim \|g\|_{L^2}\|f\|_{B_{\infty,1}^{s}}\lesssim\|g\|_{L^2}\|f\|_{W^{s+\epsilon,\infty}}.
\end{aligned}
\end{equation}
	
The desired result \eqref{appendix:8} follows from \eqref{appendix:9} and \eqref{appendix:10}.

\end{proof}

We recall the composition estimate whose proof could be found in \cite{MR2304160} or \cite{MR2768550}.
\begin{lemma}\label{composition estimate}
Let $h:\mathbb{R}\rightarrow \mathbb{R}$ a smooth function with $h(0)=0.$ Suppose $u\in H^s(\mathbb{R}^3)\cap L^{\infty}(\mathbb{R}^3)(s>0),$ then
$h(u)\in H^s(\mathbb{R}^3),$
and the following holds:

\begin{equation}\label{appendix:19}
    \begin{aligned}
      \|h(u)\|_{H^s(\mathbb{R}^3)}\leq
      C(s,|h|_{C^{[s]+1}},\|u\|_{L^{\infty}})\|u\|_{H^s(\mathbb{R}^3)}.
    \end{aligned}
\end{equation}
\end{lemma}
 If in addition, $h'(0)=0,$ then %
 \begin{equation}\label{appendix:20}
     \begin{aligned}
     \|h(u)\|_{H^s(\mathbb{R}^3)}\leq
      C(s,|h|_{C^{[s]+1}},\|u\|_{L^{\infty}})\|u\|_{L^{\infty}(\mathbb{R}^3)}\|u\|_{H^s(\mathbb{R}^3)}.
     \end{aligned}
 \end{equation}
and  for any $u,v\in H^s(\mathbb{R}^3)\cap L^{\infty}(\mathbb{R}^3)(s>0),$ one has that:
\begin{equation}\label{appendix:21}
     \begin{aligned}
     \|h(u)-h(v)\|_{H^s(\mathbb{R}^3)}\leq C(s,|h|_{C^{[s]+1}},\|(u,v)\|_{L^{\infty}})\|u-v\|_{L^{\infty}\cap H^s} \|(u,v)\|_{L^{\infty}\cap H^s}.
     \end{aligned}
    \end{equation}
\begin{corollary}\label{esofF}
Recall $F$ is defined in \eqref{defofF},
we have for $s>5/2,$
\begin{equation}
    \|F\|_{H^s}\leq C(\|(\vr,n)\|_{H^{s+1}}) (\|n\|_{W^{1,\infty}}+\|\vr\|_{W^{s+3,\infty}}) (\|(n,\nabla n)\|_{H^s}.
\end{equation}
\begin{proof}
 We will control the term $F$ by product \eqref{appendix:2.5}-\eqref{appendix:2.75} and composition estimates \eqref{appendix:19}-\eqref{appendix:21}.
 For instance, by product estimate \eqref{appendix:2.75} and composition estimates \eqref{appendix:20},
 \begin{equation*}
\begin{aligned}
\|(g(\rho)-g(\bar{\rho}))\nabla n\|_{H^s}& \leq g(\bar{\rho})\|(\vr+n)\nabla n\|_{H^s}+\|h(\vr+n)\nabla n\|_{H^s}\\
&\lesssim \|\na n\|_{H^s}\|n\|_{L^{\infty}}+\|\na n\|_{L^{\infty}}\|n \|_{H^s}+\|\vr\|_{W^{s,\infty}}\|\nabla n\|_{H^s}\\
&\qquad  +\|h(\vr+n)\|_{L^{\infty}}\|\na n\|_{H^s}+\|h(\vr+n)\|_{H^{s}}\|\nabla n\|_{L^{\infty}}\\
&\lesssim (\|(n,\na n)\|_{L^{\infty}}+\|\vr\|_{W^{s,\infty}})\|(n,\nabla n)\|_{H^s}.
\end{aligned}
\end{equation*}
 where $h(y)=g(\bar{\rho}+y)-g(\bar{\rho})-g'(\bar{\rho})y$ satisfies $g(0)=g'(0)=0.$

 For the term
 $$\nabla \big((K(\rho)-K(\tilde{\rho}))\Delta\vr\big)=\big(K(\rho)-K(\tilde{\rho})\big)\nabla\Delta\vr +\nabla \big(K(\rho)-K(\tilde{\rho})\big)\Delta\tilde{\vr}=\colon (1)+(2),$$
 We only estimate $(1)$ as $(2)$ is similar.
 Denote $h_1(x)=K(\bar{\rho}+x)-K(\bar{\rho})-K'(\bar{\rho})x,$ we have that:
 \begin{equation}
 (1)=\big(K'(\bar{\rho})n+h_1(\varrho+n)-h_1(\varrho)\big)\nabla\Delta\varrho
 \end{equation}
 We thus have by \eqref{appendix:2.5} and \eqref{appendix:21}:
 \begin{equation}
 \begin{aligned}
   \|(1)\|_{H^s}&\lesssim (\|n\|_{H^s}+ \|h_1(\varrho+n)-h_1(\varrho)\|_{H^s})\|\vr\|_{W^{s+3,\infty}}\\
   &\lesssim (1+\|(\vr,n)\|_{H^s\cap L^{\infty}})
   \|n\|_{H^s\cap L^{\infty}}\|\vr\|_{W^{s+3,\infty}}\lesssim \|n\|_{H^s}\|\vr\|_{W^{s+3,\infty}}.
  \end{aligned}
 \end{equation}
 The other terms in the expression of $F$ can be controlled in the same manner, we omit the proof.
 \end{proof}
\end{corollary}

\section*{Acknowledgement}{The author would like to thank his supervisor Professor Fr\' ed\' eric Rousset
for his kind guidance and encouragements.
He thanks Professor Corentin Audiard for the fruitful discussions in the conference 'Inhomogeneous Flows' held in CIRM.
He would also send his appreciation to Yuexun Wang for his careful checking of the details and for his useful suggestions which improve the presentation greatly.}

\bibliographystyle{abbrv}
\nocite{*}
\bibliography{ref}

\end{document}